\documentclass[12pt]{amsart}

\usepackage[latin1]{inputenc}
\usepackage[english]{babel}
\usepackage{amsthm}
\usepackage{amsmath}
\usepackage{amsxtra}
\usepackage{mathrsfs}
\usepackage{amssymb}
\usepackage[dvips]{graphicx}
\usepackage{graphicx}

\newcommand{\calS}{\mathcal{S}}
\newcommand{\calT}{\mathcal{T}}
\newcommand{\calM}{\mathcal{M}}
\newcommand{\calN}{\mathcal{N}}

\newcommand{\bbM}{\mathbb{M}}

\newcommand{\bbQ}{\mathbb{Q}}

\newcommand{\restr}{\upharpoonright}
\newcommand{\forces}{\Vdash}

\newcommand{\dom}{{\rm dom}}

\newcommand{\urltilde}{\kern -.15em\lower .7ex\hbox{~}\kern .04em}

\newcommand\M{\mathcal{M}}

\newcommand\axiom{\mathrm}

\newcommand\PFA{\axiom{PFA}}

\newcommand\ZFC{\axiom{ZFC}}

\newcommand\ORD{\axiom{ORD}}

\theoremstyle{plain}
\newtheorem{thm}{Theorem}[section]
\newtheorem{lem}[thm]{Lemma}
\newtheorem{prop}[thm]{Proposition}
\newtheorem{cor}[thm]{Corollary}
\newtheorem{fact}[thm]{Fact}
\theoremstyle{definition}
\newtheorem{defn}[thm]{Definition}

\newtheorem{claim}[thm]{Claim}
\theoremstyle{remark}
\newtheorem{rem}[thm]{Remark}

\DeclareMathOperator{\otp}{{\rm o.t.}}

\numberwithin{equation}{section}

\begin{document}

\title[PFA and the nonstationary ideal]{PFA and precipitousness of the nonstationary ideal} %NS$_{\omega_1}$}
%\author{S. Shelah}
%\address{Institute of Mathematics, Hebrew University,
%Jerusalem, Israel
%and
%Department of Mathematics, Rutgers University, New Brunswick, NJ, USA}
%\email{shelah@math.huji.ac.il}
%\urladdr{http://shelah.logic.at}
%\thanks{The first author would like to thank the Israel Science Foundation for partial support of this research. Publication XXX}

%\author{l. Tuomi}
%\address{Equipe de Logique Math\'ematique,
%Institut de Math\'ematiques de Jussieu,
%Universit\'e Paris Diderot, Paris, France}
%\email{tuomi@math.univ-paris-diderot.fr}

\author{B. Veli\v ckovi\'c}
\address{Equipe de Logique Math\'ematique,
Institut de Math\'ematiques de Jussieu,
Universit\'e Paris Diderot, Paris, France}
\email{boban@math.univ-paris-diderot.fr}
\urladdr{http://www.logique.jussieu.fr/~ boban}
%\date{2003-07-01}

\begin{abstract}
We apply Neeman's method of forcing with side conditions
to show that PFA does not imply the precipitousness of
the nonstationary ideal on $\omega_1$.
\end{abstract}

\maketitle

\section*{Introduction}

One of the main consequences of Martin's Maximum (MM) is that  the nonstationary ideal
on $\omega_1$ (${\rm NS}_{\omega_1}$) is saturated and hence also precipitous. This was already shown
by Foreman, Magidor and Shelah in \cite{FMS}, where the principle MM was  introduced.
It is natural to ask if the weaker Proper Forcing Axiom (PFA) is sufficient to imply
the same conclusion.  In the 1990s the author adapted the argument of Shelah
in \cite[Chapter XVII]{Sh_P}, where PFA is shown to be consistent with the existence of a function
$f:\omega_1 \to \omega_1$ dominating all the canonical functions below $\omega_2$, to
show that PFA does not imply the precipitousness of ${\rm NS}_{\omega_1}$.
This result was also obtained independently by Shelah and perhaps several other people,
but since it was never published it was considered a folklore result in the subject.

Some twenty years later  Neeman \cite{Neeman} introduced a method for iterating proper forcing by using conditions
which consist of two components: the \textit{working part}, which is a function of finite support, and the \textit{side condition},
which is a finite $\in$-chain of models of one of two types. The interplay between the working parts and
the side conditions allows us to show that the iteration of proper forcing notions is proper.
Neeman used this new iteration technique to give another proof of the consistency of PFA
as well as several other interesting applications.

In this paper we adapt Neeman's iteration technique to give another proof of the consistency
of PFA together with ${\rm NS}_{\omega_1}$ being non precipitous. Our modification consists of
two parts. First, we consider a {\em decorated} version of the side condition poset. This version
is already present in \cite{Neeman}. Its principal virtue is that it guarantees that the generic
sequence of models added by the side condition part of the forcing is continuous.
The second modification is more subtle. To each condition $p$ we attach the {\em height function}
${\rm ht}_p$ which is defined on certain pairs of ordinals. In order for a condition $q$ to extend $p$
we require that ${\rm ht}_q$ extends ${\rm ht}_p$. Now, if $G$ is a generic filter, we can define
the derived height function ${\rm ht}_G$ from which we can read off the functions
$h_\alpha$, for $\alpha <\theta$, where $\theta$ is the length of the iteration.
Each of these functions is defined on a club in $\omega_1$ and takes values in $\omega_1$.
If ${\mathcal U}$ is a generic ultrafilter over ${\mathcal P}(\omega_1)/{\rm NS}_{\omega_1}$
we can consider $[h_\alpha]_{\mathcal U}$, the equivalence class of $h_\alpha$ modulo $\mathcal U$,
as an element of ${\rm Ult}(V,{\mathcal U})$. The point is that the family
$\{ [h_\alpha]_{\mathcal U}: \alpha <\theta\}$ cannot be well ordered by $\leq_{\mathcal U}$
and hence ${\rm Ult}(V,{\mathcal U})$ will not be well-founded.
Now, our requirement on the height functions introduces some complications in the proof
of Neeman's lemmas required to show the properness of the iteration. The main change concerns
the pure side condition part of the forcing. Our side conditions consist of pairs
$(\M_p,d_p)$ where $\M_p$ is the $\in$-chain of models and $d_p$ is the decoration.
If a model $M$ occurs in $\M_p$ we can form the restriction $p|M$, which is simply $(\M_p\cap M,d_p\restriction M)$.
Then $p|M$ is itself a side condition and belongs to $M$.
The problem is that we do not know that $p$ is stronger than $p|M$, simply because
 ${\rm ht}_p$ may not extend ${\rm ht}_{p|M}$. However, for many models $M$ we will be
able to find a {\rm reflection} $q$ of $p$ inside $M$ such that ${\rm ht}_q$ does extend
${\rm ht}_p \restriction M$ and we will then use $q$ instead of $p |M$.
This requires reworking some of the lemmas of \cite{Neeman}. Since our main changes involve
the side condition part of the forcing, we present a detailed proof that after forcing
with the pure side condition poset the nonstationary ideal on $\omega_1$ is not precipitous. 
When we add the working parts we need to rework some of Neeman's iteration lemmas, but the modifications
are  mostly straightforward, so we only sketch the arguments and the refer the reader to \cite{Neeman}. 
Finally, let us mention that precipitousness of ideals in forcing extensions was studied by
Laver \cite{Laver}, and while we do not use directly results from that paper, some of our ideas
were inspired by \cite{Laver}.

The paper is organized as follows. In \S 1 we recall some preliminaries about canonical functions
and precipitous ideals. In \S 2 we introduce a modification of the pure side condition forcing with models
of two types from \cite{Neeman}. In \S 3 we prove a factoring lemma for our modified pure side
condition poset and use it to show that after forcing with this poset the nonstationary ideal
is non precipitous. In \S 4 we introduce the working parts and show how to complete the proof of
the main theorem.% As mentioned above, we avoid repeating the arguments from \cite{Neeman}.

In order to read this paper, a fairly good understanding of \cite{Neeman} is necessary and
the reader will be referred to it quite often. Our notation is fairly standard and can be found in
\cite{Sh_P} and \cite{Jech} to which we refer the reader for background information on precipitous ideals,
proper forcing, and all other undefined concepts. Let us just mention that a family $\mathcal F$ of subsets
of a set $K$ is called {\em stationary} in $K$ if for every function $f:K^{<\omega}\to K$ there is $M\in \mathcal F$
which is closed under $f$, ie. such that $f[M^{<\omega}]\subseteq M$.

\section{Preliminaries}

We start by recalling the relevant notions concerning precipitous ideals from \cite{JMMP}.
Suppose $\mathcal I$ is a $\kappa$-complete ideal on a cardinal $\kappa$ which contains all singletons.
Let ${\mathcal I}^+$ be the collection of all $\mathcal I$-positive subsets of $\kappa$,
i.e. ${\mathcal I}^+={\mathcal P}(\kappa)\setminus \mathcal I$.
We consider ${\mathcal I}^+$ as a forcing notion under inclusion. If $G_{\mathcal I}$ is a $V$-generic
over ${\mathcal I}^+$ then $G_{\mathcal I}$ is an ultrafilter on ${\mathcal P}^V(\kappa)$ which
extends the dual filter of $\mathcal I$. We can then form the generic ultrapower ${\rm Ult}(V,G_{\mathcal I})$
of $V$ by $G_{\mathcal I}$ in the usual way, i.e. it is simply $(V^\kappa \cap V)/G_{\mathcal I}$.
Recall that $\mathcal I$ is called \textit{precipitous} if the maximal condition forces that this ultrapower
is well founded. There is a convenient reformulation of this property in terms of games.

\begin{defn}\label{game} Let $\mathcal I$ be a $\kappa$-complete ideal on a cardinal $\kappa$ which
contains all singletons. The game ${\mathcal G}_{\mathcal I}$ is played between two players
${\rm I}$ and ${\rm II}$ as follows.

\[
\begin{array}[c]{ccccccccc}
{\rm I} : & E_0 \phantom{E_1} & E_2 \phantom{E_3}& \ \ \cdots \ \ & E_{2n} \phantom{E_{2n+1}}& \ \ \cdots \ \ \\
\hline
{\rm II} : & \phantom{E_0} E_1 & \phantom{E_2} E_3 & \ \ \cdots \ \ & \phantom{E_{2n}} E_{2n+1} & \ \ \cdots \ \ \\
\end{array}
\]

\noindent We require that $E_n\in {\mathcal I}^+$ and $E_{n+1}\subseteq E_n$, for all $n$.
The first player who violates these rules loses. If both players respect the rules, we say that
${\rm I}$ wins the game if $\bigcap _n E_n=\emptyset$. Otherwise, ${\rm II}$ wins.
\end{defn}

\begin{fact}[Galvin, Jech, Magidor \cite{GJM}]\label{game char} The ideal $\mathcal I$ is precipitous
if and only if player ${\rm I}$ does not have a winning strategy in ${\mathcal G}_{\mathcal I}$.
\qed
\end{fact}

We will also need the notion of \textit{canonical functions} relative to the nonstationary ideal, ${\rm NS}_{\omega_1}$.
We recall the relevant definitions from \cite{GaHa}.
Given $f,g: \omega_1 \rightarrow {\rm ORD}$ we let $f<_{{\rm NS}_{\omega_1}}g$ if $\{ \alpha : f(\alpha)<g(\alpha)\}$
contains a club. Since ${\rm NS}_{\omega_1}$ is countably complete, the quasi order $<_{{\rm NS}_{\omega_1}}$ is well founded.
For a function $f\in {\rm ORD}^{\omega_1}$, let $||f||$ denote the rank of $f$ in this ordering.
It is also known as the \textit{Galvin-Hajnal norm} of $f$.
By induction on $\alpha$, the $\alpha$-th \textit{canonical function} $f_\alpha$ is defined (if it exists)
as the $<_{{\rm NS}_{\omega_1}}$-least ordinal valued function greater than the $f_\xi$, for all $\xi <\alpha$.
Clearly, if the $\alpha$-th canonical function exists then it is unique up to the equivalence
$=_{{\rm NS}_{\omega_1}}$. One can show in {\rm ZFC} that the $\alpha$-th canonical function $f_\alpha$
exists, for all $\alpha <\omega_2$. One way to define $f_\alpha$ is to fix an increasing
continuous sequence $(x_\xi)_{\xi < \omega_1}$ of countable sets with $\bigcup_{\xi< \omega_1}x_\xi = \alpha$
and let $f_\alpha(\xi)= {\rm o.t.}(x_\xi)$, for all $\xi$.
The point is that if we wish to witness the non well foundedness of the generic ultrapower
we have to work with functions that are above the $\omega_2$ first canonical functions.
Our forcing is designed to introduce $\theta$ many such functions, where $\theta$ is the length of
the iteration. From these functions we  define a winning strategy for I in $\mathcal G_{\rm NS_{\omega_1}}$
and implies that ${\rm NS}_{\omega_1}$ is not precipitous in the final model.

\section{The Side Condition Poset}

We start by reviewing Neeman's side condition poset from \cite{Neeman}.
We fix a transitive model $\mathcal K=(K,\in,\ldots)$ of a sufficient fragment of ZFC, possibly
with some additional functions or predicates.
Let $\calS$ denote a collection of countable elementary submodels of $\mathcal K$
and let $\calT$ be a collection of transitive $W\prec \mathcal K$ such that $W\in K$.
We say that the pair $(\calS,\calT)$ is {\em appropriate} if $M\cap W\in \calS \cap W$, for every
$M\in \calS$ and $W\in \calT$. We are primarily interested in the case when $K$ is equal to $V_\theta$,
for some inaccessible cardinal $\theta$, let  $\calS$  consists of all countable submodels of
$V_\theta$ and $\calT$ consists of all the $V_\alpha$ such that $V_\alpha \prec V_\theta$
and $\alpha$ has uncountable cofinality. We present the more general version since it will be needed
in the analysis of the factor posets of the side condition forcing.

Let us fix a transitive model $\mathcal K$ of a sufficient fragment of set theory
and an appropriate pair $(\calS,\calT)$.
The \textit{side condition poset} $\bbM_{\calS , \calT}$ consists of finite $\in$-chains
$\calM = \{M_0 , \ldots , M_{n-1} \}$ of elements of $\calS \cup \calT$, closed under intersection.
So for each $k < n$, $M_k \in M_{k+1}$, and if $M,N \in \calM$, then also $M \cap N \in \calM$.
We will refer to elements of $\calM \cap \calS$ as \textit{small} or \textit{countable} \textit{nodes}
of $\calM$, and to the elements of $\calM \cap \calT$ as \textit{transitive} nodes of $\calM$.
We will write $\pi_{\calS}(\calM)$ for $\calM \cap \calS$ and $\pi_{\calT}(\calM)$ for $\calM \cap \calT$.
Notice that $\calM$ is totally ordered by the ranks of its nodes, so it makes sense to say,
for example, that $M$ is above or below $N$, when $M$ and $N$ are nodes of $\calM$.
The order on $\bbM_{\calS , \calT}$ is reverse inclusion, i.e. $\calM \leq \calN$ iff $\calN\subseteq \calM$.

The \textit{decorated} side condition poset $\bbM^{\rm dec}_{\calS , \calT}$ consists of pairs
$p$ of the form $(\calM_p , d_p)$, where $\calM_p \in \bbM_{\calS , \calT}$ and
$d_p:\calM_p \to K$ is such that $d_p(M)$ is a finite set which belongs
to the successor of $M$ in $\calM_p$, if this successor exists, and if $M$ is the largest node
of $\calM_p$ then $d_p(M)\in K$. Sometimes our $d_p$ will be only a partial function on $\M_p$.
In this case, we identify it with the total function which assigns the empty set to all nodes
on which $d_p$ is not defined.  The order on $\bbM^{\rm dec}_{\calS , \calT}$ is given by letting $q \leq p$ iff
$\calM_p \subseteq \calM_q$ and $d_p(M) \subseteq d_q(M)$, for every $M \in \calM_p$.
Suppose $p \in \bbM^{\rm dec}_{\calS , \calT}$ and $Q \in \calM_p$. Let $p \vert Q$ be
the condition $(\calM_p \cap Q , d_p \restr Q)$. One can check that
$p \vert Q$ is indeed a condition in $\bbM^{\rm dec}_{\calS , \calT}$.

We first observe the following simple fact.

\begin{lem}\label{M_p} Suppose $p$ is a condition in $\bbM^{\rm dec}_{\calS , \calT}$
and $M$ is a model in $\calS \cup \calT$ such that $p\in M$. Then there is a  condition $p^M$
extending $p$ such that $M$ is the top model of $\M_{p^M}$.
\end{lem}
\begin{proof} We let $\M_{p^M}$ be the closure of $\M_p \cup \{ M\}$ under intersection.
Note that if $M\in \calT$ then all the nodes of $\M_p$ are subsets of $M$ and hence
$\M_{p^M}$ is simply $\M_p \cup \{ M\}$. On the other hand if $M$ is countable we need
to add nodes of the form $M\cap W$, where $W$ is a transitive mode in $\M_p$.
We define $d_{qp^M}$ by letting $d_{p^M}(N)=d_p(N)$, if $N\in \M_p$, and $d_{p^M}(N)=\emptyset$,
if $N$ is one of the new nodes. It is straightforward to check that $p^M$ is as desired.
\end{proof}

 The main technical results about the (decorated) side condition poset are Corollaries 2.31 and 2.32
 together with Claim 2.38 in \cite{Neeman}. We combine them here as one lemma.

\begin{lem}[\cite{Neeman}]\label{scl} Let $p$ be a condition in $\bbM^{\rm dec}_{\calS , \calT}$, and let $Q$
be a node in $\calM_p$. Suppose that $q$ is a condition in $\bbM^{\rm dec}_{\calS , \calT}$
which belongs to $Q$ and strengthens the condition $p \vert Q$. Then there is
$r \in \bbM^{\rm dec}_{\calS , \calT}$ with $r \leq p,q$ such that:
\begin{enumerate}
\item $\calM_r$ is the closure under intersection of $\calM_p \cup \calM_q$,
\item $\calM_r \cap Q = \calM_q$,
\item The small nodes of $\calM_r$ outside $Q$ are of the form $M$ or
$M \cap W$, where $M$ is a small node of $\calM_p$ and $W$ is a transitive node of $\calM_q$.
\end{enumerate}
\qed
\end{lem}

We now discuss our modification of Neeman's posets. Let $\theta =K\cap \ORD$. We will choose
$\calS$ and $\calT$ to be stationary families of subsets of $K$. The stationary of $\calS$
guarantees that $\omega_1$ is preserved in the generic extension and the stationarity of $\calT$
guarantees that $\theta$ is preserved. All the cardinals in between will be collapsed to $\omega_1$,
thus $\theta$ becomes $\omega_2$ in the final model.
We plan to simultaneously add $\theta$-many partial functions from $\omega_1$ to $\omega_1$.
Each of the partial functions will be defined on a club in $\omega_1$ and will be forced
to dominate the $\theta$ first canonical functions in the generic extension.
The decorated version of the pure side condition forcing gives us
a natural way to represent the canonical function $f_\alpha$, for cofinally many $\alpha <\theta$.

Before we introduce our version of the side condition poset let us make a definition.

\begin{defn}\label{g} Let $\calM$ be a member of $\bbM_{\calS,\calT}$.
We define the partial function $h_{\calM}$ from $\theta \times \omega_1$ as follows.
The domain of $h_{\calM}$ is the set of pairs $(\alpha,\xi)$ such that
there is a countable node $M\in \calM$ with $M\cap \omega_1=\xi$ and $\alpha \in M$.
If $(\alpha,\xi)\in \dom(h_{\calM})$ we let
$$
h_{\calM}(\alpha,\xi) = \max \{ \otp(M\cap \theta) : M\in \pi_{\calS}(\M), \alpha \in M \mbox{ and } M\cap \omega_1 =\xi\}.
$$
\end{defn}

If $p\in \bbM^*_{\calS,\calT}$ we let $h_p$ denote $h_{{\calM}_p}$. We are now ready
to define our modified side condition poset $\bbM^*_{\calS , \calT}$.

\begin{defn}\label{refined side cond defn} The poset $\bbM^*_{\calS , \calT}$
consists of all conditions $p \in \bbM^{\rm dec}_{\calS,\calT}$
such that
\medskip
\begin{enumerate}
\item[($\ast$)]
for every  $M,N\in \pi_{\calS}(\calM_p)$,
if $N \cap \omega_1 \in M$, then $\otp{(N \cap \theta)} \in M$.
\end{enumerate}
\medskip
\noindent The ordering is defined by letting $q \leq p$ iff
$(\calM_q , d_q) \leq_{\bbM^{\rm dec}_{\calS , \calT}} (\M_p , d_p)$
and $h_p\subseteq h_q$.
\end{defn}

Let us first observe that we have an analog of Lemma \ref{M_p}.

\begin{lem}\label{M_p*} Suppose $p$ belongs to $\bbM^*_{\calS , \calT}$
and $M$ is a model in $\calS \cup \calT$ such that $p\in M$. Then there is a condition $p^M$
extending $p$ such that $M$ is the top node of $\M_{p^M}$.
\qed
\end{lem}

We now establish some elementary properties of conditions in $\bbM^*_{\calS,\calT}$.

\begin{lem}\label{restriction} Suppose $p$ is a condition in $\bbM^*_{\calS,\calT}$
and $M\in \calM_p$. Then $h_p\restriction M \in M$.
\end{lem}

\begin{proof} If $M$ is a transitive node then this is immediate. Suppose $M$ is countable.
We will use the following.

\begin{claim}\label{small-restriction} Suppose $N$ is a countable node in $\calM_p$ and $N\cap \omega_1\in M$.
Then $N\cap M \in M$.
\end{claim}
\begin{proof} Since $\calM_p$ is closed under intersection we have that $N\cap M\in \calM_p$.
Moreover, since $N\cap \omega_1 \in M$ we have that $N\cap M$ is below $M$.
If there is no transitive node between $N\cap M$ and $M$ then $N\cap M\in M$.
Otherwise, let $W$ be the least transitive node above $N\cap M$.
By closure under intersection again, $M\cap W \in \calM_p$. Moreover, $N\cap M \subseteq W\cap M$
and the inclusion is proper. Therefore, $M\cap W$ is a countable node above $N\cap M$
and there is no transitive node between them. Therefore, $N\cap M \in M\cap W$ and
so $N\cap M \in M$.
\end{proof}

Let us say that a node $N$ of $\calM_p$ is an \textit{end node} of $\calM_p$
if there is no node in $\calM_p$ which is an end extension of $N$.
The domain of the function $h_p$ is the union of all the sets of the form $(N \cap \theta) \times \{ \xi\}$,
where $N$ is a countable end node of $\calM_p$ and $\xi =N\cap \omega_1$.
Moreover, on $(N \cap \theta)\times \{ \xi\}$ the function $h_p$ is constant and equal to $\otp (N\cap \theta)$.
Now, if $\xi \in M$ then by Claim \ref{small-restriction} $N\cap M \in M$.
Moreover, since $p\in \bbM^*_{\calS,\calT}$ we have that $\otp(N \cap \theta)\in M$.
It follows that $h_p \restriction M\in M$.
\end{proof}

We wish to have an analog of Lemma \ref{scl}. If $p\in \bbM^*_{\calS,\calT}$ and $Q\in \calM_p$
we can let $p\vert Q$ be $(\calM_p \cap Q , d_p \restr Q)$. It is easy to check that
$p \vert Q$ is a condition. However, we do not know that $p$ extends $p\vert Q$ since
$h_p$ may not be an extension of $h_{p\vert Q}$. We must refine the notion
of restriction in order to arrange this.
In order to do this, let us enrich our initial structure $\mathcal K$ by adding predicates
for $\calS$ and $\calT$. Let $\mathcal K^*$ denote the structure $(K,\in,\calS,\calT,\ldots)$.
Note that our poset  $\bbM^*_{\calS,\calT}$ is definable in $\mathcal K^*$.
Let $\calS^*$ be the collection of all $M\in \calS$ that are elementary in $\mathcal K^*$
and let $\calT^*$ be the set of all $W\in \calT$ that are elementary in $\mathcal K^*$.
Note that $\calS^*$ (respectively $\calT^*$) is a relative club in $\calS$ (respectively $\calT$),
hence if $\calS$ (respectively $\calT$) is stationary then so is $\calS^*$ (respectively $\calT^*$).

Assume $p$ is a condition and $Q$ a node in $p$ which belongs to $\calS^*\cup \calT^*$.
Now, we know that $p\vert Q$ and $h_p\restriction Q$ belong to $Q$. Moreover, $Q$ is elementary
in $\mathcal K^*$ and $\bbM^*_{\calS,\calT}$ is definable in this structure.
Therefore, there is a condition $q \in Q$ such that $\M_p \cap Q \subseteq \calM_q$,
$d_p(R)\subseteq d_q(R)$, for all $R\in \M_q\cap Q$, and $h_q$ extends $h_p\restriction Q$.
We will call such $q$ a \textit{reflection} of $p$ inside $Q$.
Note that if $q$ is a reflection of $p$ inside $Q$ then any condition $r \in Q$ which is stronger
than $q$ is also a reflection of $p$ inside $Q$. Let us say that $p$ {\em reflect} to $Q$
if $p| Q$ is already a reflection of $p$ to $Q$. Finally, let us say that $p$ is {\em reflecting}
if $p$ reflects to $W$, for all transitive nodes $W$ in $\M_p$.

We now have a version of Lemma \ref{scl} for our poset.

\begin{lem}\label{scl2} Let $p$ be a condition in $\bbM^*_{\calS , \calT}$, and let $Q$ be a node in $\calM_p$
which belongs to  $\calS^* \cup \calT^*$. Suppose that $q \in \bbM^*_{\calS , \calT}$
is a reflection of $p$ inside $Q$. Then there is $r \in \bbM^*_{\calS , \calT}$ with $r \leq p,q$ such that:
\begin{enumerate}
\item $\calM_r$ is the closure under intersection of $\calM_p \cup \calM_q$,
\item $\calM_r \cap Q = \calM_q$,
\item The small nodes of $\calM_r$ outside $Q$ are of the form $M$ or $M \cap W$,
where $M$ is a small node of $\calM_p$ and $W$ is a transitive node of $\calM_q$.
\end{enumerate}
\end{lem}

\begin{proof} Let $r$ be the condition given by Lemma \ref{scl}. We  need
to check that $\calM_r$ satisfies ($\ast$) and $h_r$ extends $h_p$ and $h_q$.
Suppose $N,M$ are countable nodes in  $\calM_r$ and $N\cap \omega_1 \in M$.
We need to check that $\otp (N)\in M$. If $N$ and $M$ are both in $\calM_p$
or $\calM_q$ this follows from the fact that $\calM_p$ and $\calM_q$ satisfy ($\ast$).
Now, suppose $N\in \calM_q$ and $M\in \calM_r \setminus \calM_q$.
Then $M$ is of the form $M' \cap W$, for some transitive node $W$ of $\calM_q$.
Since $N\in Q$ it follows that $\otp(N) \in Q$, so if $M\cap \omega_1 \geq Q\cap \omega_1$,
then we have $\otp(N)\in M$. If $M\cap \omega_1 < Q\cap \omega_1$, then
by Fact \ref{small-restriction}, $M'\cap Q \in Q$ and hence $M'\cap Q \in \calM_q$,
therefore our conclusion follows from the fact that $\calM_q$ satisfies ($\ast$)
and $M\cap \omega_1 =M'\cap \omega_1$.
Now, suppose $M\in \calM_q$ and $N\in \calM_r \setminus \calM_q$.
Then $N$ is of the form $N'\cap W$, for some countable node $N'\in \calM_p$ and
a transitive node $W\in \calM_q$.
Since $N\cap \omega_1 \in M$ and $M\in Q$ it follows that $N'\cap \omega_1 \in Q$.
By our assumption $q$ is a reflection of $p$ inside $Q$, so there is a node $N''\in \calM_q$
such that $N''\cap \omega_1 = N'\cap \omega_1$ and $\otp(N'')=\otp(N')$.
Now, $N''$ and $M$ both belong to $\calM_q$ which satisfies ($\ast$), so $\otp(N'')\in M$.
Therefore, in all cases $\otp(N)\in M$.

Now, we check that $h_r$ extends $h_p$ and $h_q$.
Since every node in $\calM_r$ is either in $\calM_q$ or is of the form $M\cap W$,
for some countable node $M$ of $\calM_r$ and transitive $W\in \calM_q$, and
$q$ is a reflection of $p$ inside $Q$ it follows that the set of the end nodes of
$\calM_r$ is precisely the union of the end nodes of $\calM_p$ and the end nodes
of $\calM_q$. This implies that $h_r$ extends $h_p$ and $h_q$.
This completes the proof of the lemma.
\end{proof}

We have a couple of immediate corollaries.

\begin{cor}\label{fully reflect} For every condition $p\in \bbM^*_{\calS,\calT}$ there is a reflecting
condition $q\leq p$ which has the same top model as $p$.
%The set of reflecting conditions  is dense in $\bbM^*_{\calS,\calT}$.
\qed
\end{cor}

\begin{cor}\label{sp} $\bbM^*_{\calS , \calT}$ is $\calS^* \cup \calT^*$-strongly proper.
In particular, if $\calS$ is stationary then $\bbM^*_{\calS , \calT}$ preserves $\omega_1$,
and if $\calT$ is stationary then $\bbM^*_{\calS , \calT}$ preserves $\theta$.

\qed
\end{cor}

From now on, we assume that $\calS$ and $\calT$ are stationary families of subsets of $K$.
Suppose that $G$ is a $V$-generic filter over $\bbM^*_{\calS , \calT}$.
Let $\M_G$ denote $\bigcup \{ \M_p : p \in G\}$.
Then $\M_G$ is an $\in$-chain of models in $\calS \cup \calT$. Hence, $\M_G$ is totally
ordered by $\in^*$, where $\in^*$ denotes the transitive closure of  $\in$.
If $M,N$ are members of $\M_G$ with $M\in^* N$ let $(M,N)_G$ denote the interval consisting
of all $P \in \M_G$ such that $M\in^*P \in^*N$.
The following lemma is the main reason we are working with the decorated version
of the side condition poset.

\begin{lem}\label{continuity} Suppose $W$ and $W'$ are two consecutive elements
of $\M_G \cap \calT$. Then the $\in$-chain $(W,W')_G$ is continuous.
\end{lem}

\begin{proof} Note that $(W,W')_G$ consists entirely of countable models
and therefore the membership relation is transitive on $(W,W')_G$.
Suppose $M$ is a limit member of $(W,W')_G$. We need to show that
$M$ is the union of the $\in$-chain $(W,M)_G$. Let $p\in G$ be a condition such that $M\in \M_p$
and $p$ forces that $M$ is a limit member of $(W,W')_G$.
Given any $x\in M$ and a condition $q\leq p$ we show that there is $r\leq q$ and
a countable node that $R\in \M_r\cap M$ such that $x\in R$.
We may assume that there is a countable model in $\M_q$ between $W$ and $M$.
Let $Q$ be the $\in^*$-largest such model. By increasing $d_q(Q)$ if necessary,
we may assume that $x\in d_q(Q)$. Since $p$ forces that $M$ is a limit member of $\M_G$
so does $q$. Therefore, there exists $r \leq q$ such that $\M_r$ contains
a countable node between $Q$ and $M$. Let $R$ be the $\in^*$-least such node.
By the definition of the order relation on $\bbM^*_{\calS, \calT}$
we must have that $d_q(P)\in R$ and hence $x\in R$, as desired.
\end{proof}

Note that Lemma \ref{continuity} implies in particular that if $W'$ is a successor element
of $\M_G \cap \calT$ then $W'$ has cardinality $\omega_1$ in $V[G]$. Therefore, if $\beta =W'\cap \ORD$,
one way to represent the $\beta$-th canonical function $f_\beta$ in $V[G]$  is the following.
Let $W$ be the predecessor of $W'$ in $\M_G\cap \calT$. Since $\calS$ is stationary in $K$,
so is $\calS \cap W'$ in $W'$. Therefore, $(W,W')_G$ will be an $\in$-chain of length $\omega_1$.
Let $\{M_\xi : \xi <\omega_1\}$ be the increasing enumeration of this chain.
Then we can let $f_\beta(\xi)=\otp(M_\xi \cap \beta)$, for all $\xi$.
Note that $M_\xi \cap \omega_1 =\xi$, for club many $\xi$.

Now, let $h_G$ denote $\bigcup \{ h_p: p\in G\}$. Then $h_G$ is a partial function
from $\theta \times \omega_1$ to $\omega_1$. Let $h_{G,{\alpha}}$ be a partial function
from $\omega_1$ to $\omega_1$ defined by letting $h_{G,{\alpha}}(\xi)=h_G(\alpha,\xi)$,
for every $\xi$ such that $(\alpha,\xi)\in \dom (h_G)$.
By Lemma \ref{continuity} and the above remarks we have the following.

\begin{cor}\label{dominating} For every $\alpha <\theta$ the function $h_{G,\alpha}$ is defined on a club in $\omega_1$.
Moreover, $h_{G,\alpha}$ dominates under $<_{{\rm NS}_{\omega_1}}$ all the canonical function $f_\beta$,
for $\beta <\theta$.

\qed
\end{cor}

\section{Factoring the Side Condition Poset}

We now let $\mathcal K =(V_\theta,\in,\ldots)$, for some inaccessible cardinal $\theta$.
Let $T$ be the set of all $\alpha <\theta$ of uncountable cofinality such that $V_\alpha \prec \mathcal K$
and let $\calT=\{ V_\delta : \delta \in T\}$. Finally, let $\calS$ be the set of all countable elementary
submodels of $\mathcal K$. Clearly, the pair $(\calS,\calT)$ is appropriate. Let $\calS^*$ and $\calT^*$
be defined as before and let $T^*=\{ \alpha : V_\alpha \in \calT^*\}$.
We start by analyzing the factor posets of  $\bbM^*_{\calS,\calT}$.
Suppose $\delta \in T^*$ and let $p_\delta=(\{ V_\delta\} ,\emptyset)$. Then, by Lemma \ref{scl2},
the map $i_\delta: \bbM^*_{\calS, \calT}\cap V_\delta \to \bbM^*_{\calS , \calT}\restriction p_\delta$
given by $i_\delta (p) = (\M_p\cup \{ V_\delta \},d_p)$ is a complete embedding.
Fix a $V$-generic filter $G_\delta$ over $\bbM^*_{\calS, \calT}\cap V_\delta$. Let $\M_{G_\delta}$ denote
$\bigcup \{ \M_p: p\in G_\delta\}$ and let $h_{G_\delta}$ be the derived height function,
i.e. $h_{G_\delta}=\bigcup \{ h_p: p\in G_\delta\}$.
%An {\em end node} of $\M_{G_\delta}$ is a countable node $M\in \M_{G_\delta}$ which has no end extensions in $\M_{G_\delta}$.
Let $\bbQ_\delta$ denote the factor forcing $\bbM^*_{\calS, \calT}\! \! \restriction \! p_\delta/i_\delta[G_\delta]$.
We can identify $\bbQ_\delta$ with the set of all conditions $p \in \bbM^*_{\calS, \calT}$
such that $V_\delta\in \M_p$, $p$ reflects to $V_\delta$ and $p | V_\delta \in G_\delta$.

We make the following definition in $V[G_\delta]$.

\begin{defn}\label{S-delta} Let $\calS_\delta$ be the collection of all  $M\in \calS$ such that $M\nsubseteq V_\delta$,
$M\cap V_\delta \in \M_{G_\delta}$ and ${\rm o.t.}(M\cap \theta) \leq h_{G_\delta}(\alpha,M\cap \omega_1)$,
for all $\alpha \in M\cap \delta$.
\end{defn}

We also let  $T_\delta=T\setminus (\delta+1)$ and $\calT_\delta=\{ V_\gamma: \gamma \in T_\delta\}$.
We define $\calS_\delta^*$ and $\calT_\delta^*$ as before.
Clearly, the pair $(\calS_\delta,\calT_\delta)$ is appropriate. We show that $\bbQ_\delta$ is very
close to $\bbM^*_{\calS_\delta,\calT_\delta}$. More precisely, let $\bbM^*_\delta$ consist
of all pairs $p$ of the form $(\M_p,d_p)$ such that $\M_p \in \bbM_{\calS_\delta,\calT_\delta}$,
$d_p: \M_p \cup \{ V_\delta\}\to V_\theta$, $(\M_p,d_p\restriction \M_p)\in \bbM^*_{\calS_\delta,\calT_\delta}$,
and $V_\theta$ and $d_p(V_\theta)$ belong to the least model of $\M_p$. So, formally we do not put
$V_\delta$ as the least node of conditions $p$ in $\bbM^*_\delta$, but we require the function $d_p$ to be defined
on $\M_p \cup \{ V_\delta\}$. This puts a restriction on the nodes we are allowed to add below the least node of $\M_p$.

% Let $\bbM^*_\delta$ be the set of all conditions
%$p \in \bbM^*_{\calS_\delta,\calT_\delta}$ such that $V_\delta$ belongs to the least model of $\M_p$.

\begin{lem} $\bbQ_\delta$ and $\bbM_\delta^*$ are equivalent forcing notions.
\end{lem}

\begin{proof} Given a condition $p\in \bbQ_\delta$, let $\varphi(p)=(\M_p \setminus V_{\delta+1},d_p \restriction (\M_p\setminus V_{\delta}))$.
Clearly, the function $\varphi$ is order preserving. To see that $\varphi$ is onto, let $s\in \bbM_\delta^*$.
Then $M\cap V_\delta \in \M_{G_\delta}$, for all small nodes $M\in \M_s$.
Fix a condition $p\in G_\delta$ such that $M\cap V_\delta \in \M_p$, for every such $M$.
Define a condition $q$ by letting $\M_q= \M_p \cup \{ V_\delta\} \cup \M_s$ and $d_q=d_p \cup d_s$.
Since every small node of $\M_s$ is in $\calS_\delta$ it follows that ${\rm ht}_q\restriction \delta \times \omega_1 = {\rm ht}_p$.
Therefore, $q \in \bbQ_\delta$ and $\varphi(q)=s$.
Finally, note that  if $p,q\in \bbQ_\delta$ then $p$ and $q$ are compatible in $\bbQ_\delta$  iff $\varphi(p)$ and $\varphi(q)$ are
compatible in $\bbM_\delta^*$. This implies that $\bbQ_\delta$ and $\bbM_\delta^*$ are equivalent forcing notions.
\end{proof}

\begin{cor}\label{delta-strong-properness} $\bbQ_\delta$ is $\calS_\delta^*\cup \calT_\delta^*$-strongly proper.
\qed
\end{cor}

%The following lemma is stated in $V[G_\delta]$.

\begin{lem}\label{stationary}   $\calS_\delta$ is stationary family of countable subsets of $V_\theta$.
\end{lem}

\begin{proof} We argue in $V$ via a density argument. Let $\dot{f}$ be a $\bbM^*_{\calS, \calT}\cap V_\delta$-name for a function
from $V_\theta^{<\omega}$ to $V_\theta$ and let $p\in \bbM^*_{\calS, \calT}\cap V_\delta$. We find
a condition $q \leq p$ and $M\in \calS$ such that $q$ forces that $M$ belongs to $\dot{\calS}_\delta$ and is closed
under $\dot{f}$. For this purpose, fix a cardinal $\theta^*>\theta$ such that $V_{\theta^*}$ satisfies a sufficient fragment of ZFC.
Let $M^*$ be an countable elementary submodel of $V_{\theta^*}$
containing all the relevant parameters. It follows that $M\in \calS^*$, where $M=M^* \cap V_\theta$.
Let $p^M$ be the condition given by Lemma \ref{M_p*}. Since $\delta \in \calT^*$ we can find
a reflection $q$ of $p^M$ inside $V_\delta$. We claim that $q$ and $M$ are as required.
To see this,  note that, since $M\cap V_\delta \in \M_q \cap \calS^*$,
then, by Lemma \ref{scl2}, $q$ is $(M\cap V_\delta,\bbM^*_{\calS, \calT}\cap V_\delta)$-strongly generic
and hence also $(M^*, \bbM^*_{\calS, \calT}\cap V_\delta)$-generic. It follows that $q$ forces
that $M^*[\dot{G}_\delta]\cap V_\theta=M$, and hence that $M$ is closed under $\dot{f}$.
On the other hand, $h_{p^M}(\alpha,M\cap \omega_1)={\rm o.t.}(M\cap \theta)$, for all $\alpha \in M\cap \theta$.
Since $q$ is a reflection of $p^M$, we have $h_q(\alpha,M\cap \omega_1)=h_{p^M}(\alpha,M\cap \omega_1)$,
for all $\alpha \in M\cap \delta$. Therefore, $q$ forces $M$ to belong to $\dot{\calS}_\delta$.
This completes the argument.
\end{proof}

We need to understand which stationary subsets of $\omega_1$ in $V[G_\delta]$ will remain stationary in
the final model. So, suppose $E$ is a  subset of $\omega_1$ in $V[G_\delta]$.
Let
\[ \calS_\delta(E) = \{ M \in \calS_\delta : M \cap \omega_1 \in E \}.
\]

\noindent For $\rho \in T_\delta$ let $\calS_\delta^\rho(E)= \calS_\delta(E)\cap V_\rho$. Note that if $M\in \calS_\delta(E)$ and
$\rho \in T_\delta$ then $M\cap V_\rho \in \calS_\delta^\rho(E)$.
Therefore, if $\rho <\sigma$ and $\calS_\delta^{\sigma}(E)$ is stationary in $V_{\sigma}$
then $\calS_\delta^\rho(E)$ is stationary in $V_{\rho}$. Since $\theta$ is inaccessible,
it follows that  $\calS_\delta(E)$ is stationary in $V_\theta$ iff $\calS_\delta^{\rho}(E)$
is stationary in $V_\rho$, for all $\rho \in T_\delta$.

\begin{lem}\label{characterization} The maximal condition in $\bbQ_\delta$ decides if $E$
remains stationary in $\omega_1$. Namely, if $\calS_\delta(E)$ is stationary in $V_\theta$
then $\forces_{\bbQ_\delta} \check{E} \textrm{ is stationary}$, and if $\calS_\delta(E)$
is nonstationary then $\forces_{\bbQ_\delta} \check{E} \textrm{ is nonstationary}$.

%i.e. $E$ remains a stationary subset of $\omega_1$ in $V[G]$, for any $V$-generic filter
%$G$ over $\mathbb M^*_{\calS,\calT}$ extending $G_\delta$ such that $V_\delta \in \M_G$.
\end{lem}

\begin{proof} The first implication follows from Corollary \ref{delta-strong-properness}
and the fact that $\calS_\delta^*$ is a relative club in $\calS_\delta$. For the second implication,
suppose $\calS_\delta(E)$ is nonstationary and fix a successor element of $T_\delta$, say $\sigma$,
such that $\calS_\delta^{\sigma}(E)$ is nonstationary in $V_\sigma$.
Let $\rho$ be the predecessor of $\sigma$ in $T_\delta$ and fix a condition $p\in \bbQ_\delta$
such that $V_\rho,V_\sigma \in \M_p$. Pick an arbitrary $V[G_\delta]$-generic filter $G$ over $\bbQ_\delta$
containing $p$. Then we can identify $G$ with a $V$-generic filter $\bar{G}$ over $\mathbb M^*_{\calS,\calT}$ which
extends $G_\delta$ and such that $V_\delta \in \M_{\bar{G}}$.
Since $p\in \bar{G}$, we have that $V_\rho$ and $V_\sigma$ are consecutive elements of $\M_{\bar{G}}\cap \calT$.
By Lemma \ref{continuity} we know that, in $V[\bar{G}]$, $\calS_\delta \cap V_\sigma$ contains a club of countable subsets
of $V_\sigma$. On the other hand, by our assumption, $\calS_\delta^\sigma(E)$ is nonstationary.
It follows that $E$ is a nonstationary subset of $\omega_1$ in $V[\bar{G}]$. Since $G$ was an arbitrary
generic filter containing $p$, it follows that $p \forces_{\bbQ_\delta} \check{E} \mbox{ is nonstationary in } \omega_1$.

\end{proof}

\begin{rem} One can show that if $\delta$ is inaccessible in $V$ then $\bbQ_\delta$ actually preserves stationary subsets of $\omega_1$.
To see this note that, under this assumption, for every subset  $E$ of $\omega_1$ in $V[G_\delta]$ there is $\delta^*< \delta$ with
$V_{\delta^*}\in \M_{G_\delta}$ such that  $E\in V[G_{\delta^*}]$, where $G_{\delta^*}= G_\delta \cap V_{\delta^*}$.
If, in the model $V[G_{\delta^*}]$, $\calS_{\delta^*}(E)$ is  nonstationary  there is $\rho <\theta$ such that
$\calS_{\delta^*}^\rho(E)$ is nonstationary. By elementarity of $V_\delta$ in $V_\theta$ there is such $\rho <\delta$.
But then, as in the proof of Lemma \ref{characterization}, we would have that  $E$ is nonstationary already
in the model $V[G_\delta]$.
\end{rem}

Suppose $E\in V[G_\delta]$ is a subset $\omega_1$ and $\gamma <\delta$. Let
\medskip
\[
\calS_\delta(E,\gamma) =
\{ M \in \calS_\delta(E): \gamma,\delta \in M \mbox{ and } {\rm o.t.}(M\cap \theta) < h_{G_\delta}(\gamma,M\cap \omega_1)\}.
\]

\medskip
\noindent Recall that if $M\in \calS_\delta$ then $M\cap V_\delta \in G_\delta$.
Hence, if $\gamma \in M$ then $(\gamma,M\cap \omega_1) \in \dom (h_{G_\delta})$.

\begin{lem}\label{gamma} Suppose that, in $V[G_\delta]$, $E$ is a subset of $\omega_1$ such that
 $\calS_\delta(E)$ is stationary. Then $\calS_\delta(E,\gamma)$ is stationary, for all $\gamma <\delta$.
\end{lem}

\begin{proof} Work in $V[G_\delta]$ and let $\gamma <\delta$ and $f: V_\theta^{<\omega} \rightarrow V_\theta$ be given.
We need to find a member of $\calS_\delta(E,\gamma)$ which is closed under $f$.
Since $\theta$ is inaccessible, we can first find $\sigma \in T_\delta$ such that $V_\sigma$ is closed under $f$.
We know that $\calS_\delta(E)$ is stationary, hence we can  find $M\in \calS_\delta(E)$ which is  closed
under $f$ and such that $\gamma,\delta,\sigma \in M$. It follows that $M\cap V_\sigma$ is also closed
under $f$. Since $\sigma \in M$ we have that ${\rm o.t.}(M\cap \sigma) < {\rm o.t.}(M\cap \theta)$.
Since $M\in \calS_\delta(E)$  and $\gamma \in M$ we have that ${\rm o.t.}(M\cap \theta)\leq h_{G_\delta}(\gamma,M\cap \omega_1)$.
Finally, $(M\cap V_\sigma)\cap \omega_1 = M\cap \omega_1$. It follows that $M\cap V_\sigma \in \calS_\delta(E,\gamma)$,
as desired.
\end{proof}

We now consider what happens in the final model $V[G]$, where $G$ is $V$-generic over $\mathbb M_{\calS,\calT}^*$.
For an ordinal $\gamma <\theta$ let $D_\gamma$ denote the domain of $h_{G,\gamma}$.
Recall that, by Corollary \ref{dominating}, $D_\gamma$ contains a club, for all $\gamma$.
Given a subset $E$ of $\omega_1$ and $\gamma,\delta <\theta$ let
\[
\varphi (E,\gamma,\delta) =\{ \xi \in E\cap D_\gamma \cap D_\delta : h_{G,\delta}(\xi) < h_{G,\gamma}(\xi)\}.
\]

\begin{lem}\label{shrinking} Let $G$ be $V$-generic over $\mathbb M_{\calS,\calT}^*$. Suppose, in $V[G]$,
that $E$  is a stationary subset of $\omega_1$ and $\gamma <\theta$. Then there is $\delta <\theta$
such that $\varphi(E,\gamma,\delta)$ is stationary.
\end{lem}

\begin{proof} Since $\mathbb M^*_{\calS,\calT}$ is $\calT^*$-proper, we can find $\delta \in T^*\setminus (\gamma +1)$
such that $V_\delta \in \M_G$ and $E\in V[G_\delta]$, where $G_\delta = G\cap V_\delta$.
Since $E$ remains stationary in $V[G]$, it follows that, in $V[G_\delta]$, $\calS_\delta(E)$ is stationary.
Work for a while in $V[G_\delta]$. We claim that the maximal condition in $\bbQ_\delta$ forces that
$\dot{\varphi}(E,\gamma,\delta)$ is stationary, where $\dot{\varphi}(E,\gamma,\delta)$ is the canonical name for $\varphi(E,\gamma,\delta)$.
To see this fix a $\bbQ_\delta$-name $\dot{C}$ for a club in $\omega_1$ and a condition $p\in \bbQ_\delta$.
Let $\theta^*>\theta$ be such that $(V_{\theta^*},\in)$ satisfies a sufficient fragment of $\ZFC$.
We know, by Lemma \ref{gamma}, that $\calS_\delta(E,\gamma)$ is stationary, so we can find a countable elementary
submodel $M^*$ of $V_{\theta^*}$ containing all the relevant objects such that $M\in \calS_\delta(E,\gamma)$,
where $M=M^* \cap V_\theta$. Let $q$ be the condition $p^M$ as in Lemma \ref{M_p*} (or rather its version
for $\bbQ_\delta$). Since $\dot{C}\in M^*$ and $q$ is $(M^*,\bbQ_\delta)$-generic, it follows
that $q$ forces that $M\cap \omega_1$ belongs to $\dot{C}$.
Also, note that the top model of $\M_q$ is $M$. Hence $h_q(\delta,M\cap\omega_1) = {\rm o.t.}(M\cap \theta)$.
Since $M\in \calS_\delta(E,\gamma)$ we have that ${\rm o.t.}(M\cap \theta) <h_{G_\delta}(\gamma,M\cap \omega_1)$.
It follows that $q$ forces that $M\cap \omega_1$ belongs to the intersection of $\dot{\varphi}(E,\gamma,\delta)$
and $\dot{C}$, as required.
\end{proof}

We now have the following conclusion.

\begin{thm}\label{pure-precipitous} Let $G$ be $V$-generic over $\mathbb M^*_{\calS , \calT}$. Then, in $V[G]$,
$\theta =\omega_2$ and ${\rm NS}_{\omega_1}$ is not precipitous.
\end{thm}

\begin{proof} We already know that $\mathbb M^*_{\calS,\calT}$ is $\calS^* \cup \calT^*$-strongly proper.
This implies that $\omega_1$ and $\theta$ are preserved. Moreover, by Lemma \ref{continuity}, we know
that all cardinals between $\omega_1$ and $\theta$ are collapsed to $\aleph_1$. Therefore, $\theta$
becomes $\omega_2$ in $V[G]$.
In order to show that ${\rm NS}_{\omega_1}$ is not precipitous we describe a winning strategy $\tau$
for Player I in $\mathcal G_{{\rm NS}_{\omega_1}}$. On the side, Player I will pick a sequence $(\gamma_n)_n$ of ordinals $<\theta$.
So, Player I starts by playing $E_0=\omega_1$ and letting $\gamma_0=0$.
Suppose, in the $n$-th inning, Player II has played a stationary set $E_{2n+1}$.
Player I applies Lemma \ref{shrinking} to find $\delta <\theta$ such that $\varphi(E_{2n+1},\gamma_n,\delta)$
is stationary. He then lets $\gamma_{n+1} =\delta$ and plays $E_{2n+2}= \varphi(E_{2n+1},\gamma_n,\gamma_{n+1})$.
Suppose the game continues $\omega$ moves and II respects the rules. We need to show that $\bigcap_n E_n$ is empty.
Indeed, if $\xi \in \bigcap_n E_n$ then $\xi \in D_{\gamma_n}$, for all $n$, and
$h_{G,\gamma_0}(\xi) > h_{G,\gamma_1}(\xi) >  \ldots$ is an infinite decreasing sequence of ordinals,
a contradiction.
\end{proof}

\section{The Working Parts}

In this section we show how to add the working part to the side condition poset described in \S 2, which
allows us to define a Neeman style iteration.  As in \cite{Neeman}, if at each stage we choose a proper forcing,
the resulting forcing notion will be proper as well.  By a standard argument, if we use the Laver function to guide
our choices,  we we obtain $\PFA$ in the final model.
The point is that the relevant lemmas from \S 2 and \S 3 go through almost verbatim and hence we obtain, as before, 
that ${\rm NS}_{\omega_1}$ will be non precipitous in the final model. 

Let us now recall the iteration technique from \cite{Neeman}. We fix an inaccessible cardinal $\theta$
and a function $F:\theta \to V_\theta$. Let $\mathcal K$ be the structure $(V_\theta,\in,F)$.
Let $\calS$ be the set of all countable elementary submodels of $\mathcal K$ and $T$ the set of all
$\alpha <\theta$ of uncountable cofinality such that $V_\alpha$ is an elementary submodel of $\mathcal K$.
Let $\calT=\{ V_\alpha : \alpha \in T\}$. Define $\calS^*$, $T^*$ and $\calT^*$ as before.
Note that if $\alpha \in  T^*$ then $T^*\cap \alpha$ is definable in $\mathcal K$ from parameter $\alpha$.
Hence, if $\M\in \calS$ and $\alpha \in T^*$ then $M\cap V_\alpha \in \calS^*$.
We will define, by induction on $\alpha \in T^*\cup \{ \theta\}$, a forcing notion $\mathbb P_\alpha$.
In general, $\mathbb P_\alpha$ consists of tripes $p$ of the form $(\M_p,d_p,w_p)$ such that
$(\M_p,d_p)$ is a reflecting condition in $\mathbb M^*_{\calS,\calT}$ and $w_p$ is a finite partial
function from $T^*\cap \alpha$ to $V_\alpha$ with some properties.
If $\alpha <\beta$ are in $T^*\cup \{ \theta\}$ and $p \in \mathbb P_\beta$ we let $p \restriction \alpha$ denote
$(\M_p\cap V_\alpha, d_p \restriction (\M_p \cap V_\alpha), w_p\restriction V_\alpha)$.
It will be immediate from the definition that $p \restriction \alpha \in \mathbb P_\alpha$.
Moreover, since $(\M_p,d_p)$ is reflecting, it will be an extension of $(\M_p \cap V_\alpha, d_p \restriction (\M_p\cap V_\alpha))$.
For $\alpha \in T^*$ we will also be interested in the partial order $\mathbb P_\alpha \cap V_\alpha$.
We let $\dot{G_\alpha}$ denote the canonical $\mathbb P_\alpha \cap V_\alpha$-name for the generic filter.
If $M \in \calS\cup \calT$ and $\alpha \in M$ we let $M[\dot{G}_\alpha]$ be the canonical $\mathbb P_\alpha \cap V_\alpha$-name
for the model $M[G_\alpha]$, where $G_\alpha$ is the generic filter.
If $F(\alpha)$ is a $\mathbb P_\alpha \cap V_\alpha$-name which is forced by the maximal condition to be a proper forcing
notion we let $\dot{\mathbb F}_\alpha$ denote $F(\alpha)$; otherwise let $\dot{\mathbb F}_\alpha$ denote
the $\mathbb P_\alpha \cap V_\alpha$-name for the trivial forcing. Let $\leq_{\mathbb F_\alpha}$ be the name for the ordering on $\dot{\mathbb F}_\alpha$.

We are now ready for the main definition.

\begin{defn} Suppose $\alpha \in T^*\cup \{ \theta\}$. Conditions in $\mathbb P_\alpha$ are triples $p$ of the form
$(\M_p,d_p,w_p)$ such that:
\begin{enumerate}
\item $(\M_p,d_p)$ is a  reflecting condition in $\mathbb M^*_{\calS,\calT}$,
\item $w_p$ is a finite function with domain contained in the set $\{ \gamma \in T^* \cap \alpha : V_\gamma \in \M_p\}$,
\item if $\gamma \in \dom(w_p)$ then:
\begin{enumerate}
\item $w_p(\gamma)$ is a canonical $\mathbb P_\gamma \cap V_\gamma$-name for an element of $\dot{\mathbb F}_\gamma$,
\item if $M\in \calS \cap \M_p$ and $\gamma \in M$ then
\[ p \restriction \gamma \forces_{{\mathbb P}_\gamma \cap V_\gamma} w_p(\gamma) \mbox{ is }
(M[\dot{G_\alpha}],\mathbb F_\alpha)\mbox{-generic}.
\]

\end{enumerate}
\end{enumerate}
We let $q\leq p$ if $(\M_q,d_q)$ extends $(\M_p,d_p)$ in $\mathbb M^*_{\calS,\calT}$, $\dom(w_p)\subseteq \dom(w_q)$
and, for all $\gamma \in \dom(p)$,
\[
q\restriction \gamma \forces_{{\mathbb P}_\gamma \cap V_\gamma} w_q(\gamma)\leq_{\mathbb F_\gamma}w_p(\gamma).
\]

\end{defn}

Our posets $\mathbb P_\alpha$ is almost identical as the posets $\mathbb A_\alpha$ from \cite{Neeman}.
The difference is that we have a requirement that the height function ${\rm ht}_p$ of a condition $p$ is preserved
when going to a stronger condition and we also added the decoration $d_p$. We are restricting ourselves
to reflecting conditions $p$ since we then know that $p$ is an extension of $p \restriction \alpha$,
for any $\alpha\in T^*$ such that $V_\alpha$ is a node in $\M_p$. Of course, the working part $w_p$ is
defined only for such $\alpha$.
%The reason we are restricting ourselves to reflecting conditions is that if $p$ if reflecting and
%$V_\alpha$ is a transitive node of $\M_p$ then we know that $p$ extends $p \restriction \alpha$.
These modifications do not affect the relevant arguments from \cite{Neeman}. We state the main
properties of our posets and refer to \cite{Neeman} for the proofs.

\begin{lem}\label{strongly-proper} Suppose $\beta$ belongs to $T^*\cup \{ \theta\}$.
\begin{enumerate}
\item Let $p\in \mathbb P_\beta$ and let $V_\alpha \in \M_p \cap \calT^*$. Then $p$ is $(V_\alpha,\mathbb P_\beta)$-strongly generic.
\item Let $p\in \mathbb P_\beta$, let $V_\alpha \in \calT$ and suppose $p \in V_\alpha$. Then
$(\M_p \cup \{ V_\alpha\},d_p,w_p)$ is a condition in $\mathbb P_\beta$.
\item $\mathbb P_\beta$ is $\calT^*$-strongly proper.
\end{enumerate}
\end{lem}
\begin{proof} This is essentially the same as Lemma 6.7 from \cite{Neeman}.
\end{proof}

\begin{lem}\label{top-model} Suppose $\beta \in T^*\cup \{ \theta\}$ and $p \in \mathbb P_\beta$. Let $M\in \calS$ be such that $p\in M$.
Then there is a condition $q \in \mathbb P_\beta$ extending $p$ such that $M$ is the top model of $q$.
\end{lem}

\begin{proof} First, let $\M$ be closure of $\M_p \cup \{ M\}$ under intersection and let $d$ be the extension of $d_p$
to $\M$ defined by letting $d(N)=\emptyset$, for all $N\in \M \setminus \M_p$. Then $(\M,d)\in \mathbb M^*_{\calS,\calT}$.
By Lemma \ref{fully reflect} we can find a reflecting condition $(\M_q,d_q) \leq (\M,d)$
such that the top model of $\M_q$ is $M$.
Now, we need to define $w_q$. If $\alpha \in \dom(w_p)$ then $\mathbb P_\alpha \cap V_\alpha, \dot{\mathbb F}_\alpha \in M$.
Since $\dot{\mathbb F}_\alpha$ is forced by the maximal condition in $\mathbb P_\alpha \cap V_\alpha$ to be proper
and $w_p(\alpha)\in M$ is a canonical name for a  member of $\dot{\mathbb F}_\alpha$,
we can fix a canonical $\mathbb P_\alpha \cap V_\alpha$-name $w_q(\alpha)$ for a member of  $\dot{\mathbb F}_\alpha$
such that $p \restriction \alpha$ forces in $\mathbb P_\alpha \cap V_\alpha$ that  $w_q(\alpha)$ extends $w_p(\alpha)$
and is $(M[\dot{G_\alpha}],\dot{\mathbb F}_\alpha)$-generic. Then the condition $q=(\M_q,d_q,w_q)$ is as required.
\end{proof}

\begin{lem}\label{generic-condition} Suppose $\beta \in T^*\cup \{ \theta\}$ and $p \in \mathbb P_\beta$. Let $\theta^*>\theta$ be such that
$(V_{\theta^*},\in)$ satisfies a sufficient fragment of $\ZFC$. Let $M^*$ be a countable elementary submodel of $V_{\theta^*}$
containing all the relevant parameters. Let $M=M^* \cap V_\theta$ and suppose $M\in \M_p$.
Then $p$ is $(M^*,\mathbb P_\beta)$-generic.
\end{lem}

\begin{proof} This is essentially the same as Lemma 6.11 from \cite{Neeman}.
\end{proof}

Then, as in \cite{Neeman}, we have the following.

\begin{prop}\label{prop-PFA} Suppose that $\theta$ is supercompact and $F$ is a Laver function on $\theta$.
Let $G_\theta$ be a $V$-generic filter over $\mathbb P_\theta$.  Then $V[G_\theta]$ satisfies $\PFA$.
\qed \end{prop}

Now, if $\delta \in T^*$ and $G_\delta$ is a $V$-generic filter over $\mathbb P_\delta \cap V_\delta$, we can define
the function $h_{G_\delta}$ and the factor forcing $\mathbb Q_\delta$ as in \S 3.
Further, we define the set $\calS_\delta$ in an analogous way to Definition \ref{S-delta}
and show that it is stationary as in Lemma \ref{stationary}. We show, as in Lemma \ref{strongly-proper}
that $\bbQ_\delta$ is $\calT^*_\delta$-strongly proper. By Lemma \ref{top-model} and Lemma \ref{generic-condition}
we also get that, in $V[G_\delta]$, $\mathbb Q_\delta$ is $\calS^*_\delta$-proper.
For every subset $E$ of $\omega_1$ which belongs to $V[G_\delta]$, we define
the set $\calS_\delta(E)$ as in \S 3 and prove a version of Lemma \ref{characterization}.
Then, proceeding in the same way, for every $\gamma <\delta$ we define $\calS_\delta(E,\gamma)$
and prove an analog of Lemma \ref{gamma}. Then, turning to the final model $V[G_\theta]$, we prove
an analog of Lemma \ref{shrinking}. Finally, combining the arguments of Theorem \ref{pure-precipitous}
and Proposition \ref{prop-PFA} we get the conclusion.

\begin{thm}\label{main-thm} Suppose $\theta$ is supercompact and $F$ is a Laver function on $\theta$.
Let $G_\theta$ be $V$-generic over $\mathbb P_\theta$. Then, in  $V[G_\theta]$, $\PFA$ holds
and ${\rm NS}_{\omega_1}$ is not precipitous.
\qed
\end{thm}

\bibliographystyle{plain}
\bibliography{pfa}

\end{document}